\documentclass{article}
\usepackage{amssymb,amsmath,amsthm,graphicx,CJK}

\def\qed{\hfill {\hbox{${\vcenter{\vbox{               
   \hrule height 0.4pt\hbox{\vrule width 0.4pt height 6pt
   \kern5pt\vrule width 0.4pt}\hrule height 0.4pt}}}$}}}

\def\utr{\, \underline{\ast}\, }
\def\otr{\, \overline{\ast}\, }

\newtheorem{theorem}{Theorem}

\newtheorem{proposition}[theorem]{Proposition}

\newtheorem{conjecture}{Conjecture}
\theoremstyle{definition}
\newtheorem{example}{Example}
\newtheorem{definition}{Definition}

\date{}

\title{\Large \textbf{Bikei Homology}}

\author{Sam Nelson\footnote{Email: Sam.Nelson@cmc.edu. Partially supported by Simons Foundation collaboration grant 316709}\and
Jake Rosenfield ‎\footnote{Email: jlrosenfield@gmail.com}} 

\begin{document}
\maketitle

\begin{abstract}
We introduce a modified homology and cohomology theory for involutory 
biquandles (also known as \textit{bikei}). We use bikei 2-cocycles to enhance
the bikei counting invariant for unoriented knots and links as well as
unoriented and non-orientable knotted surfaces in $\mathbb{R}^4$.
\end{abstract}

\medskip

\parbox{5.5in} {\textsc{Keywords:} bikei, involutory biquandles, 
cocycle invariants, bikei homology

\smallskip

\textsc{2010 MSC:} 57M27, 57M25}

\section{\large\textbf{Introduction}}\label{I}

In \cite{J}, Joyce introduced an algebraic structure known as \textit{quandles}
which can be used to define computable invariants of oriented knots and links
(see also \cite{M}).
For unoriented knots and links, a special case known as \textit{involutory 
quandles} or \textit{kei} (\begin{CJK*}{UTF8}{min}圭\end{CJK*}) has been studied
going back to Takasaki \cite{T}. In \cite{FR} quandles were generalized to 
\textit{racks} and in \cite{FRS} racks were generalized to \textit{biracks}.
In \cite{AN}, the involutory case of biquandles was considered, now known as
\textit{bikei} (\begin{CJK*}{UTF8}{min}双圭\end{CJK*}).

In \cite{FRS} a homology theory for racks and biracks was introduced in which
the 2-cocycle condition corresponds to the Reidemeister III move for a certain
way of associating 2-chains to crossings in an oriented rack-colored knot or 
link diagram. In \cite{CJKLS} a subcomplex was defined corresponding to 
Reidemeister I moves in the quandle case, leading to the theory of quandle
2-cocycle invariants of knots and links. In \cite{CES} this construction was
generalized to the biquandle case. In \cite{EN1} the degenerate subcomplex was 
generalized for the case of non-quandle racks, in each case defining a new
family of cocycle enhancements of counting invariants.

In this paper we introduce a generalization of biquandle homology to the
case of bikei which we call \textit{bikei homology}. The paper is organized 
as follows. In Section \ref{B} we review the basics of bikei and the bikei 
counting invariant. In Section \ref{BH} we introduce bikei homology and 
cohomology. In Section \ref{CE} we define the bikei cocycle enhancements of 
the bikei counting invariant for unoriented knots and provide some
examples. In Section \ref{KS} we extend the bikei cocycle invariant to 
unoriented (including non-orientable) knotted surfaces in $\mathbb{R}^4$. 
In Section \ref{Q} we finish with some questions for future work.

\section{\large\textbf{Bikei}}\label{B}

We begin with a review of bikei (see \cite{AN,EN} for more).

\begin{definition}
A \textit{bikei} is a set $X$ with two binary operations 
$\utr,\otr:X\times X\to X$ satisfying for all $x,y,z\in X$
\begin{itemize}
\item[(i)] $x\utr x=x\otr x$,
\item[(ii)] 
\[\begin{array}{rclc}
(x\otr y)\otr y & = & x & (ii.i) \\
(x\utr y)\utr y & = & x & (ii.ii)\\
x\utr(y\otr x)  & = & x\utr y & (ii.iii)\\
x\otr(y\utr x)  & = & x\otr y & (ii.iv),
\end{array}\] and
\item[(iii)] (Exchange Laws) 
\[\begin{array}{rclc}
(x\otr y)\otr (x\otr y) & = & (x\otr z)\otr (y\utr z) & (iii.i) \\
(x\utr y)\otr (x\utr y) & = & (x\otr z)\utr (y\otr z)  & (iii.ii)\\
(x\utr y)\utr (z\utr y) & = & (x\utr z)\utr (y\otr z) &  (iii.iii).
\end{array}\]
\end{itemize}
Note that $x\utr y$ and $x\otr y$ are also denoted  in the literature 
by $x^y=B_2(x,y)$ and $x_y=B_1(x,y)$ respectively; we are following the 
notation used in \cite{EN}.
\end{definition} 

Some standard examples (see \cite{AN, EN}) of bikei structures include: 

\begin{example} 
Let $X$ be a set and $\sigma:X\to X$ any involution, i.e.,
any map such that $\sigma^2=\mathrm{Id}_X$. Then $X$ is a bikei with operations
\[x\utr y=x\otr y=\sigma(x)\]
known as a \textit{constant action bikei}.
\end{example}

\begin{example}
Let $\Lambda=\mathbb{Z}[t,s]/(t^2-1,s^2-1,(t-1)(s-1))$ be the quotient of the
ring of two-variable polynomials with integer coefficients such that 
$s^2=t^2=1$ by the ideal generated by $(1-t)(1-s)$. Then any $\Lambda$-module
$X$ is a bikei with operations
\[x\utr y=tx+(s-t) y,\quad x\otr y=sx\]
known as an \textit{Alexander bikei}. To see this, we can verify the axioms; we
will show (iii) and leave the verification of (i) and (ii) to the reader. 
\[\begin{array}{rcl}
(x\otr y)\otr (x\otr y) & = & 
s(sx) \\
& = & (x\otr z)\otr (y\utr z), \\
(x\utr y)\otr (x\utr y) & = & 
s(tx+(s-t)y) \\
& = & t(sx)+(s-t)(sy)\\
& = & (x\otr z)\utr (y\otr z) \\
(x\utr y)\utr (z\utr y) & = & 
t(tx+(s-t)y)+(s-t)(tz+(s-t)y) \\
& = & t^2x+t(s-t)y+t(s-t)z+(s-t)^2y \\
& = & t^2x+t(s-t)z+ (s-t)(t+s-t)y\\
& = & t(tx+(s-t)z)+(s-t)(sy)\\
& = & (x\utr z)\utr (y\otr z)
\end{array}\]
\end{example}

\begin{definition}
A map $f:X\to Y$ between bikei is a \textit{bikei homomorphism} if we have
\[f(x\utr y)=f(x)\utr f(y) 
\quad \mathrm{and}\quad
f(x\otr y)=f(x)\otr f(y) 
\]
for all $x,y\in X$. A bijective bikei homomorphism is a \textit{bikei 
isomorphism}.
\end{definition}

\begin{example}
Let $X=\{x_1,\dots, x_n\}$ be a finite set. We can represent any bikei 
structure on $X$ with an $n\times 2n$ block matrix $M$encoding the operation
tables of $\utr$ and $\otr$ by setting $M_{j,k}=l$ and $M_{j,k+n}=m$ where
$x_j\utr x_k=x_l$ and $x_j\otr x_k=x_m$ for $j,k\in\{1,\dots,n\}$. For example, 
there are two nonisomorphic bikei on the set $X=\{x_1,x_2\}$, given by
the matrices
\[
\left[\begin{array}{rr|rr}
1 & 1 & 1 & 1 \\
2 & 2 & 2 & 2 
\end{array}\right]
\quad\mathrm{and}\quad
\left[\begin{array}{rr|rr}
2 & 2 & 2 & 2 \\
1 & 1 & 1 & 1 
\end{array}\right].
\]
See \cite{AN, EN, NV} for more.
\end{example}

\begin{example} (See \cite{AN} for more)
Let $D$ be an unoriented knot or link diagram representing an unoriented knot
or link $K$ and let $G$ be a set of 
generators corresponding to semiarcs in $D$. The set $W$ of \textit{bikei 
words} in $G$ is defined recursively by the rules
\begin{itemize}
\item[(i)] $x\in G\Rightarrow x\in W$ and
\item[(ii)] $x,y\in W\Rightarrow x\utr y,x\otr y\in W$.
\end{itemize}
Then the \textit{fundamental bikei} of $D$, denoted $\mathcal{BK}(D)$, is 
the set of equivalence classes 
of $W$ under the equivalence relation generated by the bikei axioms and the
\textit{crossing relations} in $D$, i.e.
\[\includegraphics{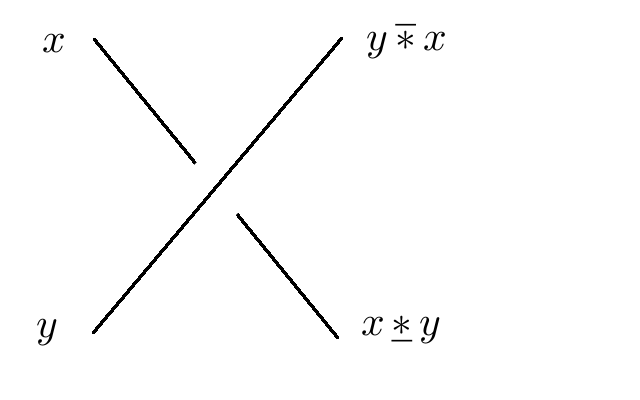}\]
We express such a bikei with a \textit{bikei presentation}, i.e. an expression
of the form
\[\mathcal{BK}(D)=\langle g_1,\dots, g_n\ | \ r_1,\dots, r_n\rangle\]
where $\{g_1,\dots,g_n\}$ are generators and $\{r_1,\dots, r_n\}$ are crossing
relations, with the bikei axiom relations understood. 
It is easy to check that Reidemeister moves on $D$ induce Tietze moves on
presentations, and hence the isomorphism type of the fundamental bikei is an 
invariant of unoriented knots and links; hence, we will generally write
$\mathcal{BK}(K)$ instead of $\mathcal{BK}(D)$.
\end{example}

Given an unoriented knot or link $K$ represented by a diagram $D$ and a finite
bikei $X$, the \textit{bikei counting invariant} $\Phi^{\mathbb{Z}}_X(K)$  is the 
cardinality of the set of bikei homomorphisms $f:\mathcal{BK}(K)\to X$, i.e.
\[\Phi^{\mathbb{Z}}_X(K)=|\mathrm{Hom}(\mathcal{BK}(K),X)|.\]
The superscript $\mathbb{Z}$ indicates this is the integer-valued 
unenhanced version of the invariant; we will soon enhance this invariant
with bikei cocycles.

Every such homomorphism assigns an element of $X$ to each generator of
$\mathcal{BK}(K)$, which we can think of as coloring the corresponding
semiarc in $D$. Conversely, an assignment of elements of $X$ to the semiarcs
in $D$ determines a bikei homomorphism $f:\mathcal{BK}(K)\to X$ only if
it satisfies the crossing relations at every crossing. Hence, we can compute
the bikei counting invariant of an unoriented knot or link by counting
bikei colorings of any diagram of $D$ which satisfy the crossing relations.

\begin{example}
Consider the bikei $X=\mathbb{Z}_2=\{0,1\}$ with $x\utr y=x\otr y=x+1$. As a 
coloring rule, this says that each time we go through a crossing either 
over or under, we switch from 0 to 1 or 1 to 0. Then for any classical
knot, there are exactly two $X$-colorings, determined by our choice of 
starting color on a choice of semiarc.
\[\includegraphics{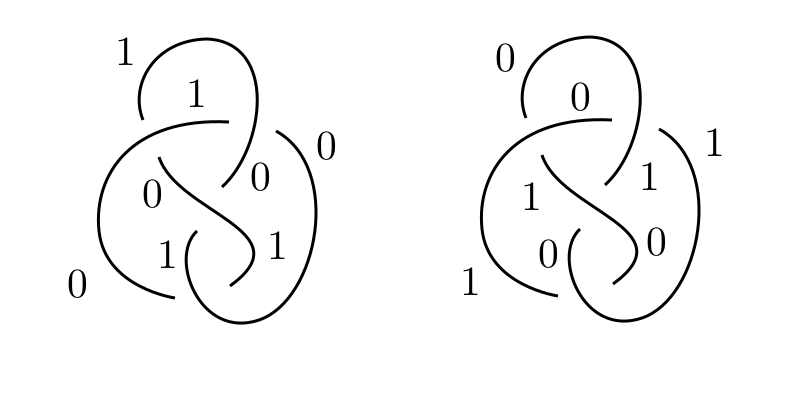}\]
\end{example}

In the next section, we will enhance the bikei counting invariant with cocycles
in a bikei homology theory to get a stronger invariant following \cite{CJKLS,CES} etc, but using bikei and unoriented diagrams.

\section{\large\textbf{Bikei Homology}}\label{BH}

\begin{definition}
Let $X$ be a bikei and $A$ and ableian group. Set $C_n(X;A)=A[X^n]$ for 
$n\ge 1$ and $\{0\}$ otherwise. The
\textit{birack boundary map} $\partial_n:C_n(X;A)\to C_{n-1}(X;A)$ 
is defined on generators $\vec{x}=(x_1,\dots,x_n)$ by
\[\partial(\vec{x})=\sum_{k=1}^n(-1)^{k-1} 
\left(\partial^1_k(\vec{x})-\partial^2_k(\vec{x})\right)\]
where
\begin{eqnarray*}
\partial^1_k(x_1,\dots,x_n) & = & (x_1,\dots,x_{k-1},x_{k+1},\dots,x_n) \ 
\mathrm{and}\\
\partial^2_k(x_1,\dots,x_n) & = & (x_1\utr x_k,\dots,x_{k-1}\utr x_k,x_{k+1}\otr x_k,\dots,x_n\otr x_k)
\end{eqnarray*}
and extended to $C_n(X;A)$ by linearity. The resulting homology and cohomolgy
groups $H_n(X;A)=\mathrm{Ker}\ \partial_n/\mathrm{Im}\ \partial_{n-1}$ and
$H^n=\mathrm{Ker}\ \delta^{n+1}/\mathrm{Im}\ \delta^n$ where 
$\delta^n \phi =\phi \partial_n$ are the \textit{birack homology and cohomology}
groups of $X$ with coefficients in $A$.
\end{definition}

In previous work \cite{CJKLS, CES, CEGN}, the subset $C^D_n(X;A)$
of $C_n(X;A)$ generated by elements $(x_1,\dots,x_n)$ with $x_j=x_{j+1}$ 
for some $j=1,\dots,n-1$ was identified as the \textit{degenerate subcomplex}.
Then the \textit{biquandle homology and cohomology} groups, also known
as the \textit{Yang-Baxter homology and cohomology} groups, are the homology 
and cohomology groups of the quotient complex $C^B_n(X;A)=C_n(X;A)/C^D_n(X;A)$.

%

We now introduce a slight generalization for bikei.
\begin{definition}
Let $X$ be a bikei. The \textit{bikei degenerate} subgroups of $C_n(X;A)$,
denoted $C^{BD}_n(X;A)$, are generated by chains of the form
$(x)-(x\utr y)$ and $(x)-(x\otr y)$ when $n=1$, chains of the forms
\[(x,x),\ (x,y)-(x\utr y,y\otr x),\ (x,y)+(x,y\otr x)\ \mathrm{and}\ 
(x,y)+(x\utr y,y)\]
for $n=2$ and by chains of the form $(\dots, x,x,\dots)$ for $n\ge 2$.
\end{definition}

\begin{proposition}
For a bikei $X$, $(C^{BD}_n,\partial_n)$ forms a subcomplex.
\end{proposition}

\begin{proof}

When $n=2$ the degenerate chains 
$(x,y)+(x\utr y,y)$ 
and $(x,y)+(x,y\otr x)$ have boundary
\begin{eqnarray*}
\partial((x,y)+(x\utr y,y)) & = &
[(y)-(y\otr x)-(x)+(x\utr y)] \\ & & \quad
+[(y)-(y\otr(x\utr y))-(x\utr y)+((x\utr y)\utr y)]  \\
& = & [(y)-(y\otr x)-(x)+(x\utr y)] 
+[(y)-(y\otr x)-(x\utr y)+(x)] \\
& = & 2[(y)-(y\otr x)],
\end{eqnarray*}
and
\begin{eqnarray*}
\partial((x,y)+(x,y\otr x)) & = &
[(y)-(y\otr x)-(x)+(x\utr y)] \\ & & \quad
+[(y\otr x)-((y\otr x)\otr x)-(x)+(x\utr(y\otr x))]  \\
& = & [(y)-(y\otr x)-(x)+(x\utr y)] 
+[(y\otr x)-(y)-(x)+(x\utr y)] \\
& = & 2[(x\utr y)-(x)]
\end{eqnarray*} while the chains  $(x,y)-(x\utr y,y\otr x)$
have boundary
\begin{eqnarray*}
\partial((x,y)-(x\utr y,y\otr x)) & = &
[(y)-(y\otr x)-(x)+(x\utr y)] \\ & & \quad
-[(y\otr x)-((y\otr x)\otr(x\utr y))-(x\utr y)+((x\utr y)\utr(y\otr x)]  \\
& = & [(y)-(y\otr x)-(x)+(x\utr y)] 
-[(y\otr x)-(y)-(x\utr y)+(x)] \\
& = & 2[(y)-(y\otr x)]-2[(x)-(x\utr y)]\in C_1^{BD}(X;A).
\end{eqnarray*}

Hence, $\partial_n(C^{BD}_n(X;A))\subset C^{BD}_{n-1}(X;A)$ and 
$(\partial,C^{BD}_n(X;A))$ is a subcomplex.
\end{proof}

\begin{definition}
Let $X$ be a bikei. Then for each $n\ge 1$, 
set $C_n^{BK}(X;A)=C_n(X;A)/C^{BD}_n(X;A)$.
The resulting homology and cohomology groups $H_n^{BK}(X;A)$ and 
$H^n_{BK}(X;A)$ are
the \textit{bikei homology and cohomology} groups of $X$.
We will generally denote elements of $C^n(X;A)$ and their
equivalence classes in $H^n_{BK}(X;A)$ by $\phi.$
\end{definition}

Since $C_n^{BD}(X;A)=C^D_n(X;A)$ for $n\ge 3$, we have:

\begin{theorem}
For $n>3$, the biquandle and bikei homology and cohomology groups for a bikei
$X$ coincide, i.e. $H_n^{B}(X;A)=H_n^{BK}(X;A)$  and $H^n_{B}(X;A)=H^n_{BK}(X;A)$.
\end{theorem}

\begin{proposition}
If $X$ is a bikei in which $x\otr y=x$ for all $x,y$ (that is, a kei) or a bikei
in which $x\utr y=x$ for all $x,y$, then $H^2_{BK}(X;\mathbb{F})=\{0\}$ for any
field  $\mathbb{F}$. 
\end{proposition}

\begin{proof}
In either of the listed cases, the degenerate group includes all cochains 
$\phi\in C^2(X;A)$ of 
the form $\phi(x,y)+\phi(x,y)=2\phi(x,y)$, so we must have $2\phi(x,y)=0$. 
Then if our coefficients belong to a field, we must have $\phi(x,y)=0$ for 
all $x,y\in X$.
\end{proof}

At first, it may seem like bikei homology is completely trivial, but it turns 
to have nontrivial torsion part in at least some cases. 

\begin{proposition}
Let $A$ be a commutative ring with identity and let $X$ be an Alexander bikei 
structure on $A$, i.e., a choice of units $t,s\in A^{\times}$ such that 
$(1-s)(1-t)=0$ defining operations
\[x\utr y=tx+(s-t)y\quad\mathrm{and}\quad x\otr y=sx.\]
Then for elements $a,b\in A$ satisfying
\[b=-a \quad \mathrm{and}\quad 2a=a(1+s)=a(1+t)=a(1-t)=a(s-t-2)=0,\] 
the linear map $\phi(x,y)=ax+by$ 
defines a bikei cocycle in $H^2_{BK}(X;A)$.
\end{proposition}

We will call such a cocycle a \textit{linear Mochizuki bikei cocycle} since 
it is similar to Mochizuki cocycles for Alexander quandles \cite{MO}.

\begin{proof}
We check the bikei degeneracy and cocycle conditions. First, 
checking the degeneracy conditions, $b=-a$  implies
$\phi(x,x)=ax+bx=(a+b)x=0$. Then setting 
$\phi(x,y)=a(x-y)$, the other degeneracy 
conditions are satisfied:
\begin{eqnarray*}
\phi(x,y)-\phi(x\utr y,y\otr x) & = & 0 \\
a(x-y)-a(tx+(s-t)y-sy) & = & 0 \\
a(1-t)x-a(-1+s-t-s)y & = & 0\\
a(1-t)x+a(1+t)y & = & 0
\end{eqnarray*}
since $a(1-t)=a(1+t)=0$,
\begin{eqnarray*}
\phi(x,y)+\phi(x\utr y,y) & = & 0 \\
a(x-y)+a(tx+(s-t)y-y) & = & 0 \\
a(1+t)x+a(-1+s-t-1)y & = & 0\\
a(1+t)x-a(s-t-2)y & = & 0
\end{eqnarray*}
since $a(s-t-2)=0$, and
\begin{eqnarray*}
\phi(x,y)+\phi(x,y\otr x) & = & 0 \\
a(x-y)+a(x-sy) & = & 0 \\
a(2)x+a(-1-s)y & = & 0\\
2ax-a(1+s)y & = & 0.
\end{eqnarray*}
Finally, we check the birack cocycle condition:
\begin{eqnarray*}
\phi(x,y)+\phi(y,z)+\phi(x\utr y, z\otr y) & = & 
\phi(x,z)+\phi(x\utr z,y\utr z)+\phi(y\otr x,z\otr x) \\
a(x-y)+a(y-z)+a(tx+(s-t)y- sz) & = & 
a(x-z)+a(tx+(s-t)z-ty-(s-t)z)+a(sy-sz) \\
a(1-t)x+a(s-t)y-a(1+s)z & = &
a(1+t)x+a(s-t)y-a(1+s)z
\end{eqnarray*}
as required.
\end{proof}

\begin{example}\label{ex:abk}
Let $X=\mathbb{Z}_8$ and set $s=3,t=1$ and $a=4$. Then we have 
$(1-s)(1-t)=0$, so $X$ is a bikei with operations 
\[x\utr y=x+2y\quad \mathrm{and}\quad x\otr y=3x. \]
Then we verify that
\[2(4)=4(1+3)=4(1+1)=4(1-1)=4(3-1-2)=0\]
and $\phi(x,y)=4x-4y$ is a nonzero bikei cocycle.
\end{example}

\section{\large\textbf{Cocycle Enhancements}}\label{CE}

Our motivation for bikei homology comes from the desire to extend cocycle 
enhancements of the bikei counting invariant to unoriented knots and links,
and in particular to non-orientable knotted surfaces in $\mathbb{R}^4$.

Let $\phi\in H^2_{BK}(X;A)$ and let $D$ be an unoriented knot or link diagram
representing a knot or link $K$. For any bikei homomorphism 
$f:\mathcal{BK}(K)\to X$, let $D_f$ denoted the $X$-coloring of $D$ determined
by $f$. That is, $f$ is a bikei homomorphism assigning elements of $X$ to
generators of $\mathcal{BK}(K)$ (which correspond one to one with semiarcs in 
$D$), while $D_f$ is the knot diagram $D$ with semiarcs labeled with
their images $f(x_j)\in X$. Then at each crossing we assign a 
\textit{Boltzmann weight} 
$\phi(x,y)$ where $x$ and $y$ are the bikei colors on the under and over 
crossing semiarcs when the crossing is positioned as depicted, with the
overstrand going from upper right to lower left.
\[\includegraphics{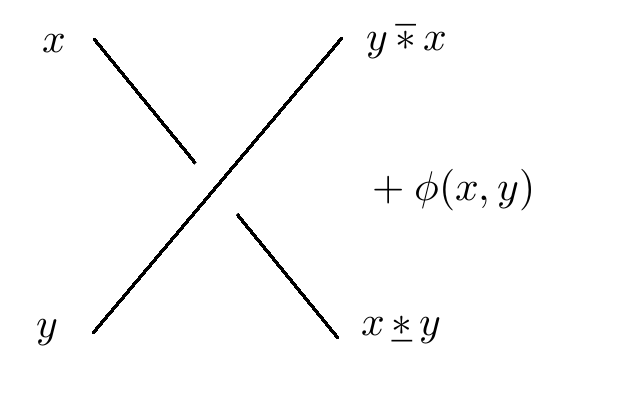}\]
Then the Boltzmann weight for the bikei coloring $f$ is the sum of the 
Boltzmann weights $\phi(x,y)$ at each crossing $C$ in the set 
$\mathcal{C}(D_f)$ of crossings in diagram $D_f$,
\[BW(f)=\sum_{C \in\mathcal{C}(D_f)} \phi(x,y).\]

The bikei 2-cocycle conditions are precisely the conditions required to ensure 
that the Boltzmann weight is unchanged by Reidemeister moves. 
The degeneracy conditions insure that the Boltzmann weight is well-defined and
give invariance under Reidemeister I and II moves:
\[\includegraphics{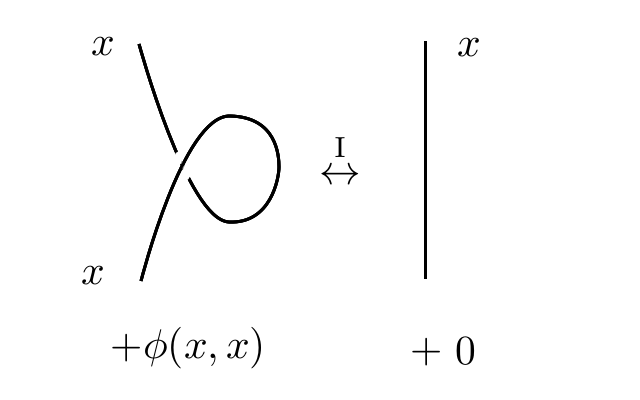} \includegraphics{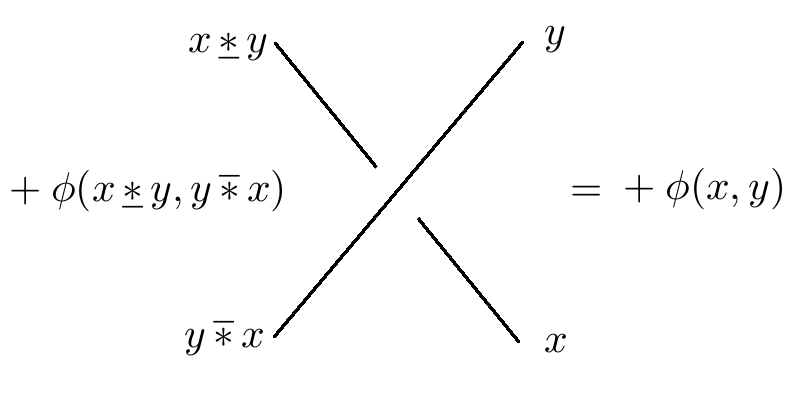}\]
\[\includegraphics{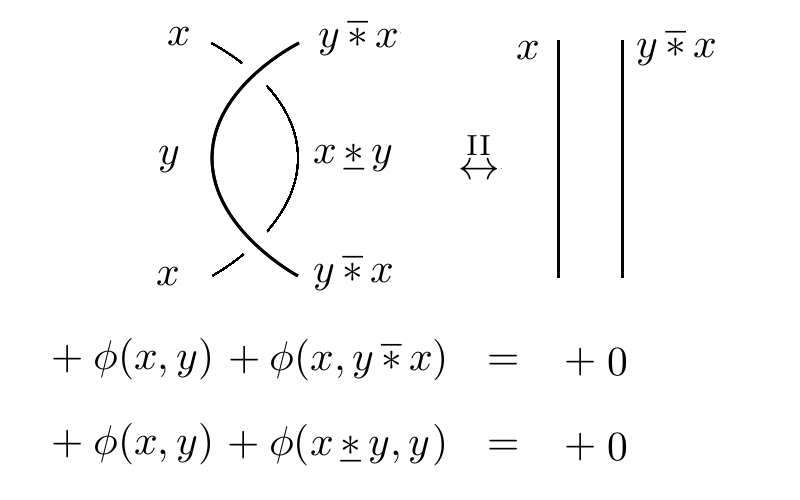}\]

The 2-cocycle condition 
\begin{eqnarray*}
\delta^2(\phi(x,y,z)) & = & 
\phi(\partial_2(x,y,z))\\
& = & \phi(-(y,z)+(x\utr y,z\otr y)
+(x,z)-(y\otr x,z\otr x)-(x,y)+(x\utr y,y\utr z)) \\
& = & 0 
\end{eqnarray*}
guarantees equivalence under Reidemeister III moves:
\[\includegraphics{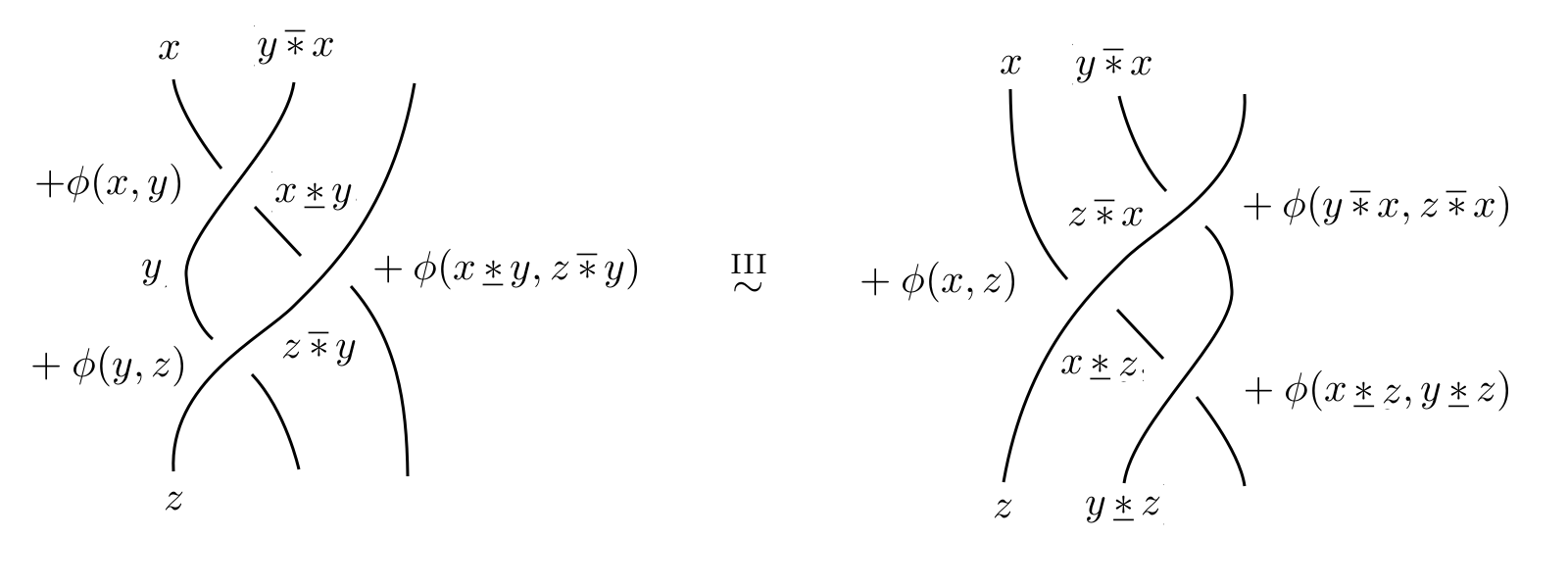}\]

Thus, we have 
\begin{definition}
Let $X$ be a finite bikei and $\phi\in H^2_{BK}(X;A)$. Then for any unoriented
knot or link $K$ represented by a diagram $D$, the \textit{2-cocycle 
enhanced bikei counting invariant} of $K$ is the multiset of Boltzmann weights
\[\Phi^{\phi,M}_X(D)=\{BW(D_f)\in A\ |\ f\in\mathrm{Hom}(\mathcal{BK}(K),X)\}.\]
We can convert this multiset into a polynomial form for ease of comparison
by making each Boltmann weight a formal exponent of a dummy variable $u$
and converting multplicities to positive integer coefficients:
\[\Phi^{\phi}_X(D)=\sum_{f\in\mathrm{Hom}(\mathcal{BK}(K),X)} u^{BW(D_f)}. \]
\end{definition}

By construction (and also see \cite{CES} etc.), we have
\begin{theorem}
If $X$ is a finite bikei, $\phi\in H^2_{BK}(X)$ and $D$ and $D'$ are 
unoriented knot or link diagrams related by Reidemeister moves, then
\[\Phi^{\phi,M}_X(D)=\Phi^{\phi,M}_X(D')
\quad\mathrm{and}\quad 
\Phi^{\phi}_X(D)=\Phi^{\phi}_X(D').\]
\end{theorem}

As with many knot invariants, we can extend $\Phi_X^{\phi}$ to the case of 
virtual knots and links by simply ignoring the virtual crossings, thinking of
our virtual knot diagram as drawn on a surface with sufficient genus to avoid 
virtual crossings; see \cite{K,CKS} for more. In particular, semiarcs do not 
end at virtual crossings.

\begin{example}\label{ex:1}
Let $X$ be the Alexander bikei from example \ref{ex:abk} and
consider the virtual Hopf link $VH$ below.
\[\includegraphics{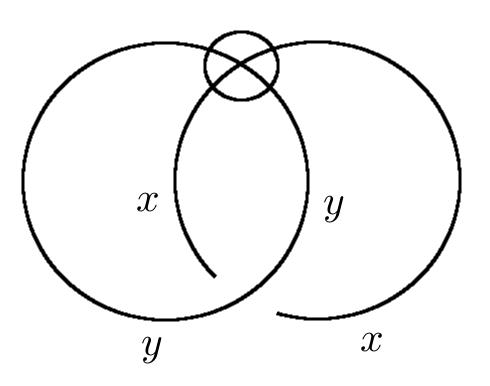}\]
We obtain system of coloring equations 
\[
\begin{array}{rcl}
x\utr y & = & x \\
y\otr x & = & y
\end{array}
\Rightarrow
\begin{array}{rcl}
x+2y & = & x \\
3y & = & y
\end{array}
\Rightarrow
2y=0
\]
so a pair $(x,y)$ yields a valid coloring for $y\in\{0,4\}$ and no further
conditions on $x\in\mathbb{Z}_8$, so there are $16$ $X$-colorings. Each coloring
has a Boltzmann weight of 
\[\phi(x,y)=4(x-y)=\left\{
\begin{array}{ll}
4 & x\ \mathrm{odd} \\
0 & x\ \mathrm{even} \\
\end{array}
\right.\] 
at the single crossing, so the invariant is $\Phi^{\phi}_X(VH)=8+8u^4$.

We can compare this with the case of the usual Hopf link
\[\includegraphics{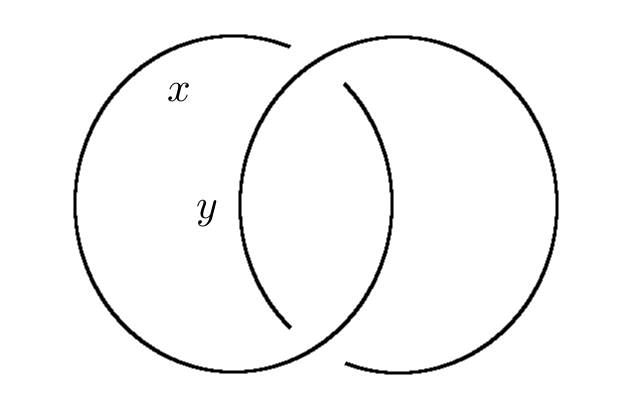}\]
where we obtain coloring equations $3x=x+2y$ and $3y=y+2x$, both reducing to 
$2x=2y$. Then there are sixteen colorings, and each has Boltzmann weight 
$\phi(x,y)+\phi(y,x)=4(x-y)+4(y-x)=0$, yielding an invariant value of 
$\Phi^{\phi}_X(H)=16$. Now, the virtual Hopf link can be distinguished from 
the classical Hopf link in other ways, e.g. the Hopf link has two colorings 
by the bikei $\mathbb{Z}_2$ with $x\utr y=x\otr y=x+1$ while the virtual 
Hopf link has none; this example demonstrates that $\Phi^{\phi}_X$
is not determined by the counting invariant and hence is a proper enhancement.
\end{example}

\section{\large\textbf{Invariants of Knotted Surfaces}}\label{KS}

We can also define bikei cocycle invariants for knotted unoriented (including 
non-orientable) surfaces in $\mathbb{R}^4$ in the same way. Recall that a 
\textit{marked graph diagram}, also called a \textit{marked vertex diagram}, is
a diagram with ordinary crossings together with \textit{saddle crossings}
\[\includegraphics{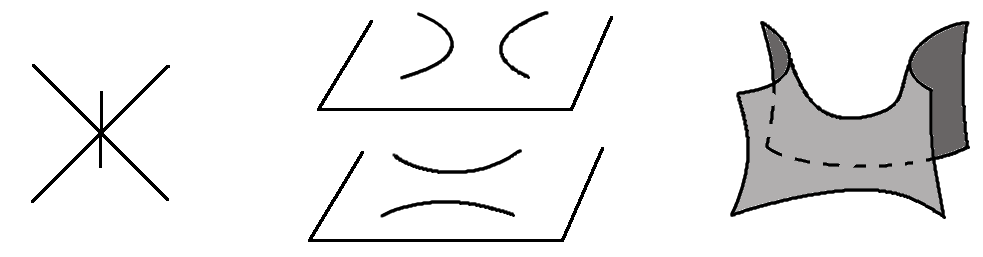}\]
representing saddle point. More precisely, given a knotted surface $\Sigma 
\subset \mathbb{R}^4$, we move the maxima in the $x_4$ direction to the 
hyperplane $x_4=1$, the minima to $x_4=-1$ and the saddle points to $x_4=0$. 
Then the intersection of $\Sigma$ with $x_4=0$ is a link diagram with 
singularities at the saddle points; we indicate the direction of the saddle
with a small bar. Such a diagram represents a knots closed surface if
both resolutions of the saddle yield unlinks; otherwise, the diagram 
represents a cobordism between the links represented by the smoothed diagrams.
Two such diagrams represent ambient isotopic knotted surfaces if and only if 
they are related by a sequence of the Reidemeister moves together with the
\textit{Yoshikawa moves}
\[\includegraphics{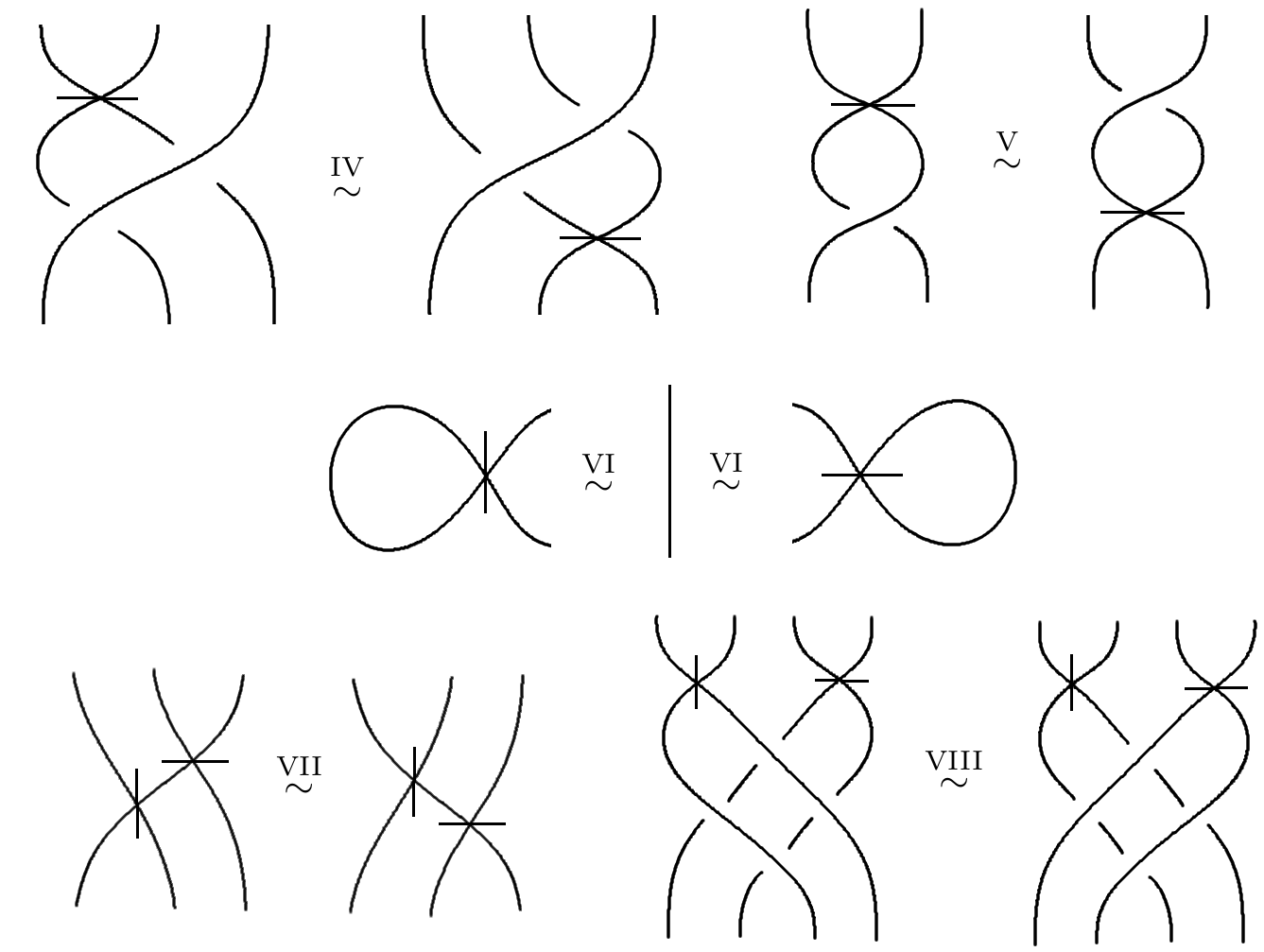}\]
See for instance \cite{CKS2,KJL2,KJL,L,LO} for more.

In \cite{NR}, bikei colorings and counting invariants of marked graph diagrams 
were considered. Specifically, all of the semiarcs meeting at a saddle crossing
determine the same generator of $\mathcal{BK}(\Sigma)$ and must have the same 
color. We now observe that we can enhance the bikei 
counting invariant with bikei $2$-cocycles in the same way as we did for
knots and links in $\mathbb{R}^3$. Specifically, we have

\begin{definition}
Let $X$ be a finite bikei and $\phi\in H^2_{BK}(X)$. Then for any unoriented
knotted surface $\Sigma$ represented by a marked vertex diagram $D$, 
the \textit{2-cocycle enhanced bikei counting invariant} of $\Sigma$ is the 
multiset  
\[\Phi^{\phi,M}_X(D)=\{BW(D_f)\ |\ f\in\mathrm{Hom}(\mathcal{BK}(\Sigma),X)\}\]
or its generating function
\[\Phi^{\phi}_X(D)=\sum_{f\in\mathrm{Hom}(\mathcal{BK}(\Sigma),X)} u^{BW(D_f)} \]
\end{definition}

We then have 
\begin{theorem}
If $X$ is a finite bikei, $\phi\in H^2_{BK}(X)$ and $D$ and $D'$ are 
unoriented marked vertex diagrams related by Yoshikawa moves, then
\[\Phi^{\phi,M}_X(D)=\Phi^{\phi,M}_X(D)
\quad\mathrm{and}\quad 
\Phi^{\phi}_X(D)=\Phi^{\phi}_X(D).\]
\end{theorem}

\begin{proof}
This is a matter of verifying that the Yoshikawa moves do not change the 
Boltzmann weight of a bikei colored marked vertex diagram. Moves VI and VII
do not involve non-saddle crossings, so these cannot change the Boltzmann 
weight, and in move V all semiarc colors are the same, so both sides of 
the move contribute $\phi(x,x)=0$. For moves IV and VIII, both sides of the 
move contribute degenerate chains: $\phi(x,y)+\phi(x,y\otr x)$ on both sides of move IV, and $2\phi(y,x)+2\phi(y,x\otr y)$ on the left and 
$2\phi(x,y)+2\phi(x,y\otr x)$ on the right of move VIII.
\[\raisebox{0.2in}{\includegraphics{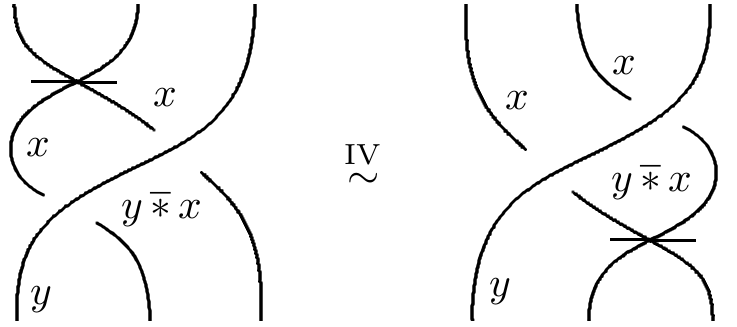}}\quad\quad
\includegraphics{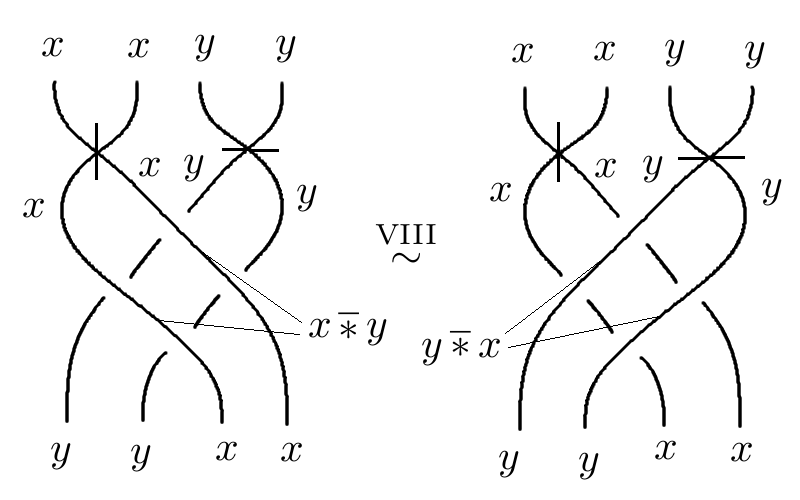}\]
\end{proof}

Analogously to the case of knotted and linked curves, including virtual 
crossings in marked vertex diagrams yields virtual knotted surface diagrams,
with the rule that two such diagrams are equivalent if related by 
Reidemeister moves, Yoshikawa moves and the detour move, i.e., redrawing 
an arc with only virtual crossings as another arc with only virtual crossings 
and the same endpoints. See \cite{K4} for more.

\begin{example}
Let $X$ again be the Alexander bikei from example \ref{ex:abk} and consider the
virtual marked vertex diagram $D$ below, representing a virtual linked surface 
with one sphere component and one projective plane component.
\[\includegraphics{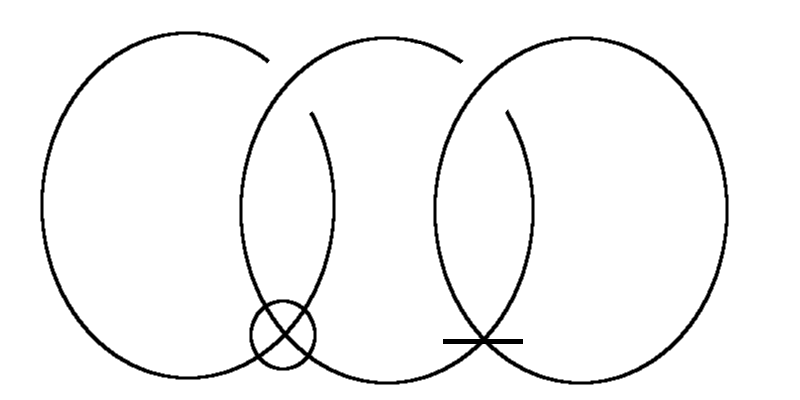}\]
$X$-labelings of $D$ are given by pairs $(x,y)\in(\mathbb{Z}_8)^2$
satisfying $3y=y$, i.e. $2y=0$, with $x=x+2y$ imposing no further 
conditions on $x$: 
\[\includegraphics{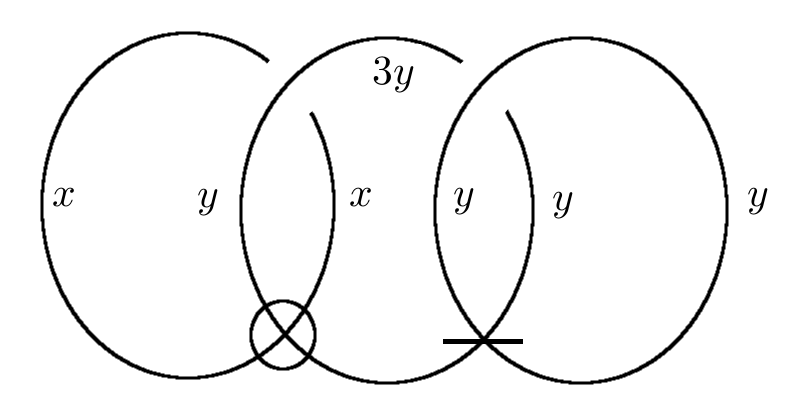}\]
Hence we have 16 $X$-colorings as in example \ref{ex:1}; each coloring has
Boltzmann weight $4(x-y)+4(y-y)$, so colorings with odd $x$ contribute $u^4$
while colorings with even $x$ contribute $1$ to the invariant, and we again 
have $\Phi^{\phi}_X(K)=8+8u^4$.
\end{example}

\section{\large\textbf{Questions}}\label{Q}

We conclude with a few questions for future research.

How does bikei homology generalize to other cases like involutory biracks, 
virtual bikei and parity bikei?  What about shadow colorings, either in the
marked vertex style or broken surface diagram style? It seems that additional
degeneracies might be required for shadow colorings, depending on the shadow 
coloring operation; what should these be?

We propose the following
\begin{conjecture}
The free part of $H^2_{BK}(X;\mathbb{Z})$ is $\{0\}$ for all finite bikei $X$.
\end{conjecture}
Is this conjecture true? If not, what is the smallest counterexample?

\bibliography{sn-jr-rev}{}
\bibliographystyle{abbrv}

\bigskip

\noindent
\textsc{Department of Mathematical Sciences \\
Claremont McKenna College \\
850 Columbia Ave. \\
Claremont, CA 91711}

\end{document}